\documentclass[12pt,leqno]{article}
\usepackage{amsfonts}
\usepackage{amsmath, amsthm, amsfonts, amssymb, color}
\usepackage{mathrsfs}\usepackage{stmaryrd}
\renewcommand{\bar}{\overline}
\renewcommand{\hat}{\widehat}
\renewcommand{\tilde}{\widetilde}

\newtheorem{thm}{Theorem}[section]

\newtheorem{lem}[thm]{Lemma}

\theoremstyle{definition}

\newcommand{\scr}[1]{\mathscr #1}
\definecolor{wco}{rgb}{0.5,0.2,0.3}

\numberwithin{equation}{section} \theoremstyle{remark}

\newcommand{\ua}{\uparrow}

\title{{\bf     Limit Theorems  for   Additive Functionals of Path-Dependent   SDEs  }
\thanks{This work is
supported   in part by NNSFC (11771326, 11431014,11831014).} }
\author{
{\bf  Jianhai Bao$^{b),c)}$,  Feng-Yu Wang$^{a),c)}$, Chenggui Yuan$^{c)}$}\\
\footnotesize{$^{a)}$Center for Applied Mathematics, Tianjin
University, Tianjin 300072, China}\\
\footnotesize{$^{b)}$School of Mathematics and Statistics, Central
South
University, Changsha 410083, China}\\
\footnotesize{$^{c)}$Department of Mathematics, Swansea University,
Bay Campus, SA1 8EN, UK}\\ \footnotesize{jianhaibao@csu.edu.cn,
wangfy@bnu.edu.cn, C.Yuan@swansea.ac.uk}}
\date{}
\begin{document}
\def\R{\mathbb R}  \def\ff{\frac} \def\ss{\sqrt} \def\B{\mathbf
B}
\def\N{\mathbb N} \def\kk{\kappa} \def\m{{\bf m}}
\def\dd{\delta} \def\DD{\Delta} \def\vv{\varepsilon} \def\rr{\rho}
\def\<{\langle} \def\>{\rangle} \def\GG{\Gamma} \def\gg{\gamma}
  \def\nn{\nabla} \def\pp{\partial} \def\EE{\scr E}
\def\d{\text{\rm{d}}} \def\bb{\beta} \def\aa{\alpha} \def\D{\scr D}
  \def\si{\sigma} \def\ess{\text{\rm{ess}}}
\def\beg{\begin} \def\beq{\begin{equation}}  \def\F{\scr F}
\def\Ric{\text{\rm{Ric}}} \def\Hess{\text{\rm{Hess}}}
\def\e{\text{\rm{e}}} \def\ua{\underline a} \def\OO{\Omega}  \def\oo{\omega}
 \def\tt{\tilde} \def\Ric{\text{\rm{Ric}}}
\def\cut{\text{\rm{cut}}} \def\P{\mathbb P} \def\ifn{I_n(f^{\bigotimes n})}
\def\C{\scr C}      \def\aaa{\mathbf{r}}     \def\r{r}
\def\gap{\text{\rm{gap}}} \def\prr{\pi_{{\bf m},\varrho}}  \def\r{\mathbf r}
\def\Z{\mathbb Z} \def\vrr{\varrho} \def\ll{\lambda}
\def\L{\scr L}\def\Tt{\tt} \def\TT{\tt}\def\II{\mathbb I}
\def\i{{\rm in}}\def\Sect{{\rm Sect}}\def\E{\mathbb E} \def\H{\mathbb H}
\def\M{\scr M}\def\Q{\mathbb Q} \def\texto{\text{o}} \def\LL{\Lambda}
\def\Rank{{\rm Rank}} \def\B{\scr B} \def\i{{\rm i}} \def\HR{\hat{\R}^d}
\def\to{\rightarrow}\def\l{\ell}
\def\8{\infty}\def\X{\mathbb{X}}\def\3{\triangle}
\def\V{\mathbb{V}}\def\M{\mathbb{M}}\def\W{\mathbb{W}}\def\Y{\mathbb{Y}}\def\1{\lesssim}

\def\La{\Lambda}\def\S{\mathbf{S}}\def\BB{\mathbb B}

\renewcommand{\bar}{\overline}
\renewcommand{\hat}{\widehat}
\renewcommand{\tilde}{\widetilde}
 \maketitle

\begin{abstract}  By using limit theorems of uniform mixing Markov processes and martingale difference sequences, the strong law of large
numbers, central limit theorem, and the law of iterated logarithm  are established for additive functionals of path-dependent stochastic differential equations.
%Since the corresponding Markov processes are highly degenerate,   known results derived for ergodic Markov processes under the $L^2$ or total variational norm do not apply.
\end{abstract}

 \vskip 2cm

\section{Introduction and Main Results}

Since  W. Doeblin \cite{Doe}  in 1938 established the law of large
numbers and central limit theorem for denumerable Markov chains,
limit theory for additive functionals of Markov processes has been
extensively investigated. In general, for an ergodic Markov process
$(X_t)_{t\ge 0}$ on a Polish space $E$, as $t\to\infty$ one
describes  the convergence of the empirical distribution $\mu_t:=\ff
1 t \int_0^t \dd_{X_s}\d s$ to the unique invariant probability
measure $\mu_\infty$. A standard  way is  to look at the convergence
rate of
$$A_t^f:= \ff 1 t \int_0^t f(X_s)\d s %= \mu_t(f)
\to \mu_\infty(f)\ \ \text{as}\ \ t\to\infty$$ for $f$ in a class of
reference functions.% determining measures.
This leads to the study of limit theorems for additive functionals
of ergodic Markov processes. Classical  limit theorems  include
\beg{itemize} \item Strong law of large numbers (SLLN): $\P$-a.s.
convergence of $A_t^f$ to $\mu_\infty(f)$;
\item Central limit theorem (CLT): The weak convergence of $\ff 1 {\ss t} \int_0^t \{f(X_s)-\mu_\infty(f)\}\d s $ to a normal random variable;
\item Law of iterated logarithm (LIL): the asymptotic  range of $\ff 1 {\ss {t\log\log t}}\int_0^t f(X_s)\d s.$\end{itemize}
Once CLT is established, one may further investigate the
large/moderate deviations principles, see for instance \cite{GF} and
references within.

When the  Markov processes  are exponentially ergodic in  $L^2(\mu_\infty)$  or total variational norm, limit theorems of $A_t^f$ have been established for
 reference functions  $f\in L^2(\mu_\infty)$ or $\B_b(E)$,
 respectively;
 see  the recent monograph \cite{Kulik} and earlier references
\cite{CCG,HH,IL,KV,JS,MT93,Wu}. However, these results do not apply
to   highly degenerate models which are  exponentially ergodic
merely under a Wasserstein distance;  see for instance \cite{HM} for
2D Navier-Stokes equations with degenerate stochastic forcing, and
\cite{BWY18,BS,B14,HMS} for stochastic differential equations (SDEs)
with memory.

In this paper, we aim to establish  limit theorems for
path-dependent SDEs, which were initiated by It\^o-Nisio \cite{IN}.
Due to the path-dependence of the noise term,  the corresponding
segment solutions are no longer   ergodic in the total variational
norm (see e.g. \cite[Example 5.1.3]{Kulik}). Moreover, the
$L^2$-ergodicity is also unknown because of the lack of Dirichlet
form for   path-dependent SDEs. So far, there are a few of papers on
LLN and CLT for stochastic dynamical systems which are weakly
ergodic; see e.g. \cite{KW,Kulik,Sh,Wal}. In particular, $f$ in
\cite{KW,Wal}  is assumed to be (bounded) Lipschitz with respect to
a metric and the weak LLN is investigated; In \cite{Kulik}, the LLN
is established under some additional technical conditions (see
\cite[Theorem 5.1.10]{Kulik} for more details). In this paper, we
will show that limit theorems established in \cite{Sh} for uniformly
mixing Markov processes apply well to the present model for   $f$
being Lipchitz continuous with respect to a quasi-metric.

For a fixed number $r_0\in (0,\infty)$, let $\C=C([-r_0,0];\R^d)$ be
the collection of all continuous functions $f:[-r_0,0]\to\R^d$
endowed with the uniform norm $$\|f\|_\8:=\sup_{-r_0\le \theta\le
0}|f(\theta)|.$$ For any continuous path $(\gg(t))_{t\ge -r_0} $ on
$\R^d$, its segment  $(\gg_t )_{ t\ge 0}$ is a continuous path on
$\C$ defined by
$$\gg_t(\theta):= \gg(t+\theta),\ \ \theta\in [-r_0,0], t\ge 0.$$
Consider the following path-dependent SDE on $\R^d$:
\begin{equation}\label{eq1}
\d X(t)=b(X_t)\d t+\si(X_t)\d W(t),~~~t\ge0,~~~X_0=\xi\in\C,
\end{equation} where $(W(t))_{t\ge0}$ is a
$d$-dimensional Brownian motion on a complete filtration probability
space $(\OO,\F,(\F_t)_{t\ge0},\P)$,  and   $$b:\C\to\R^d,\ \
 \si:\C\to\R^d\otimes\R^d$$ are  measurable maps satisfying the following assumptions.
\begin{enumerate}
\item[({\bf A1})] ({\bf Continuity}) $\si$ is Lipschitz continuous; $b$ is continuous,   and bounded on bounded subsets of
$\C$;
\item[({\bf A2})] ({\bf Dissipativity}) There exist constants $\ll_1,\ll_2>0$ with
$\ll_1>\ll_2\e^{\ll_1r_0}$ such that
\begin{equation*}
 2\<\xi(0)-\eta(0),b(\xi)-b(\eta)\>\le-\ll_1|\xi(0)-\eta(0)|^2+\ll_2\|\xi-\eta\|_\8^2,~~~~\xi,\eta\in\C;
\end{equation*}
\item[({\bf A3})] ({\bf Invertibility})  $\si$ is invertible with
$ \sup_{\xi\in\C}\{\|\si(\xi)\|+
 \|\si(\xi)^{-1}\| \}<\8$.
\end{enumerate}

Under   ({\bf A1}) and ({\bf A2}), \eqref{eq1} admits a unique
solution, and the segment (also called functional or window)
solution $(X_t)_{t\ge0}$ is a Markov process on $\C$; see
\cite[Theorem 2.2]{VS} or \cite[Proposition 4.1]{BS}. Assumption
({\bf A3}) was used in \cite{BWY18,BS,B14,HMS} to ensure the
exponential ergodicity  under the Wasserstein distance induced by a
quasi-metric.

Let $P_t$ be the associated Markov process, i.e.,
$$P_t f(\xi)= \E f(X_t^\xi),\ \ f\in \B_b(\C),~ t\ge 0,~ \xi\in \C.$$ For a probability measure $\mu$ on $\C$, let $\mu P_t$ be the law of $X_t$ with
 initial distribution $\mu$. We then have
$$\int_{\C} f\d (\mu P_t) = \int_\C P_t f \d\mu,\ \ t\ge 0, f\in \B_b(\C).$$

 To state the main results, we recall  the quasi-metric $\rho_{p,\gg}$,   the associated  Wasserstein distance $\mathbb W_{p,\gg}$,
  and the class $C_{p,\gg}(\C)$ of Lipschitz functions, where $p\ge 1$ and $\gg\in (0,1]$ are constants.
Firstly,   let
$$\rho_{p,\gg}(\xi,\eta)=(1\wedge\|\xi-\eta\|_\8^\gg)\ss{1+\|\xi\|_\8^p+\|\eta\|_\8^p},~~~\xi,\eta\in \C.$$ Note that $(\xi,\eta)\mapsto\rho_{p,\gg}(\xi,\eta)$ is a quasi-distance, i.e.,
it is symmetric, lower semi-continuous, and $ \rho_{p,\gg}(\xi,
\eta)=0\Leftrightarrow \xi = \eta$, but the triangle inequality may
not hold. Next,  let $C_{p,\gg}(\C)$   be the set of all continuous
$\R$-valued functions on $\C$  such that
\begin{equation*}
\|f\|_{p,\gg}:= \sup_{\xi\in\C}\ff{|f(\xi)|}{1+\|\xi\|_\8^{p/2}}+
  \sup_{\xi\neq\eta,\xi,\eta\in\C}\ff{|f(\xi)-f(\eta)|}{\rho_{p,\gg}(\xi,\eta)}<\8.
\end{equation*}
Moreover,  let
 $ \mathscr{P}_{p,\gg}(\C)$ be  the set of probability measures $\mu$ on $\scr C$ with $(\mu\times\mu)(\rho_{p,\gg})<\infty$. Define
\begin{equation*} \mathbb W_{p,\gg}(\mu,\nu)=\inf_{\pi\in \mathcal {C}(\mu,\nu)}
\int_{\C\times\C} \rho_{p,\gg}(\xi,\eta)
\pi(\d\xi,\d\eta),\ \ \mu,\nu \in\mathscr{P}_{p,\gg}(\C),\end{equation*} where $\mathcal {C}(\mu,\nu)$
stands for the set of all couplings of $\mu$ and $\nu$; that is,
$\pi\in\mathcal {C}(\mu,\nu)$ if and only if it is a probability measure on
$\C\times\C$ such that $\pi(\cdot\times\C)=\mu(\cdot)$ and
$\pi(\C\times\cdot)=\nu(\cdot)$.

The following  result  concerns with the exponential ergodicity and SLLN  for  the additive functional $A_t^f(\xi):=\ff 1 t \int_0^t f(X_s^\xi)\d s,$ where $ f\in C_{p,\gg}(\C)$.

\begin{thm}\label{th1}  Assume {\bf (A1)}-{\bf (A3)} and let $p\ge 1,\gg\in (0,1]$.  Then
$P_t$ has a unique invariant probability measure $\mu_\infty\in \scr
P_{p,\gg}(\C)$ such that \beq\label{EE} \mathbb W_{p,\gg}(\mu P_t,
\mu_\infty)\le c\,\e^{-\bb t} \W_{p,\gg}(\mu,\mu_\infty),\ \ t\ge 0,
\mu\in \scr P_{p,\gg}(\scr C)\end{equation} holds for some constants
$c,\bb>0$. Moreover, for any $\xi\in\C$ and $f\in C_{p,\gg}(\C)$,
\beg{enumerate}
\item[$(1)$]  There exists a constant $c>0$ such that
\begin{equation*}
\E\big|A_t^f(\xi)-\mu_\infty(f)\big|^2\le c\,(1+\|\xi\|_\8^p)\|f\|_{
p,\gg}^2t^{-1},~ t\ge1;
\end{equation*}
\item[$(2)$]  For any $\vv\in(0,\ff{1}{2}),$ there exist  a constant $c_\vv>0$ such that $\P$-a.s. \begin{equation*}
\big|A_t^f(\xi)-\mu_\infty(f)\big|\le
c_\vv\|f\|_{p,\gg}t^{-\ff{1}{2}+\vv},~~~t\ge T_\vv^f(\xi)
\end{equation*} holds for a family of  random variables $\{T_\vv^f(\xi)\ge 1: f \in C_{p,\gg}(\C),\xi\in\C\}$   satisfying
   \begin{equation*}
\sup_{ f\in C_{p,\gg}(\C)}
 \ff{\E \,|T_\vv^f(\xi)|^k}{1+\|f\|_{p,\gg}^{k(1+k)}} <\infty,~~~k\ge1.
\end{equation*}
\end{enumerate}
\end{thm}

To state the CLT, we introduce the corrector $R_f$ for   $f\in
C_{p,\gg} (\C)$ defined by
\begin{equation}\label{WWW}
R_f(\xi)=\int_0^\8\big\{P_tf(\xi)-\mu_\infty(f)\big\} \d
t,~~~\xi\in\C.
\end{equation} This function is well-defined since  \eqref{EE} and $\mu_\infty\in \scr
P_{p,\gg}(\C)$ imply
\beq\label{ABC}  \beg{split}  &|P_t f(\xi)- \mu_\infty(f)| \le \|f\|_{p,\gg} \mathbb W_{p,\gg}(\dd_\xi P_t, \mu_\infty)\\
&\le c_1\e^{-\bb t} \|f\|_{p,\gg} \W_{p,\gg}(\dd_\xi,\mu_\infty)\le
c_2 \e^{-\bb t} \|f\|_{p,\gg} (1+\|\xi\|^{p/2}_\8), \ \ t\ge
0,\,\xi\in \C\end{split} \end{equation}
 for some constants $c_1,c_2>0$.  Let
\begin{equation}\label{S7}
\varphi_f(\xi)=\E\bigg|\int_0^1f(X_r^\xi) \d r
+R_f(X_1^\xi)-R_f(\xi)\bigg|^2,~~~\xi\in\C.
\end{equation}
For any $D\in [0,\infty)$, let $\Phi_{D}$ be  the normal
distribution function with zero mean and variance $D$, where
$\Phi_0(z):= 1_{[0,\infty)}(z)$ for $D=0$. We have  the following
CLT.

\begin{thm}\label{th2} Assume {\bf (A1)}-{\bf (A3)}.  For any  constants $p\ge 1$ and $\gg\in (0,1]$, let  $f\in C_{p,\gg}  (\C)$ with $\mu_\infty(f)=0.$ Then
    $D_f:=\ss{\mu_\infty(\psi_f)}\in [0,\infty)$ and the following assertion holds:
    \beg{enumerate} \item[$(1)$] When $D_f>0$, for any $\vv\in (0,\ff 1 4)$   there exists an increasing function
$h_{\vv}:\R_+ \to\R_+$ such that
 \begin{equation*}
\sup_{z\in\R}  \Big|\P\big(\ss t A_t^f(\xi)\le z\big)-\Phi_{D_f} (z)
\Big|\le  h_{\vv}(\|\xi\|_\8)t^{-\ff{1}{4}-\vv}, \ \ t>0;
\end{equation*}
 \item[$(2)$] When $D_f=0$,    there exists an increasing function
$h_0:\R_+ \to\R_+$ such that
 \begin{equation*}
\sup_{z\in\R}  \big(1\land |z|\big)\Big|\P\big(\ss t A_t^f(\xi)\le z\big)-\Phi_{D_f} (z) \Big|\le  h_{0}(\|\xi\|_\8)t^{-\ff{1}{4}}, \ \ t>0.
\end{equation*}
\end{enumerate}
  \end{thm}

Finally, to investigate the LIL, we consider the  unit ball in the Camron-Martin space of $C([0,1];\R)$:
\begin{equation}\label{*PQ}
\mathcal {H}:=\bigg\{h\in C([0,1];\R):\ h_t' \mbox{ exists a.e.
} t,  \int_0^1|h_t'|^2\d
t \le1\bigg\},
\end{equation}
 and the following discrete version of $R_f$ and $\varphi_f$ for   $  f\in C_{p,\gg}(\C) $ with $\mu_\infty(f)=0$:
\begin{equation*}
\hat R_f(\xi):=\sum_{k=0}^\8P_kf(\xi),~~\hat\varphi_f(\xi)
:=\E\big|f(\xi)+\hat R^f(X_1^\xi)-\hat R^f(\xi)\big|^2,~~~\xi\in\C,
\end{equation*} which are  well defined due to \eqref{ABC}.
For any $n\ge 1$, consider the following random variable on $C([0,1];\R)$:
\begin{equation}\label{B1}
\Lambda_n^{f,\xi}(t):=\sum_{k=0}^n 1_{[\ff k n, \ff{k+1}n)} (t)
\ff{\sum_{l=1}^{k-1}f(X_l^\xi) +(nt-k)f(X_k^\xi)}{\hat
D_f\ss{2n\log\log n}},\ \ t\in [0,1],
\end{equation}
where $\hat D_f:= \mu_\infty(\hat\varphi_f).$

\begin{thm}\label{th3} Assume {\bf (A1)}-{\bf (A3)}. Let $p\ge 1,\gg\in (0,1]$, $\xi\in \C$, and
  $f\in C_{p,\gg}(\C)$ with $\mu_\infty(f)=0$ and $\hat D_f>0$.  Then the sequence $\{\LL_{n}^{f,\xi}(\cdot)\}_{n\ge 1}$   is almost surely relatively compact in $C([0,1];\R)$,
  and when $n\to\infty$ the set of limit  points coincides with $\mathcal{H}$.
Consequently,  $\P$-a.s.
\begin{equation}\label{B2}
\limsup_{n\to\8}\ff{\sum_{l=1}^nf(X_l^\xi)}{\ss{2n\log\log
n}}=\hat D_f ,\ \ \
\liminf_{n\to\8}\ff{\sum_{l=1}^nf(X_l^\xi)}{\ss{2n\log\log
n}}=-\hat D_f.
\end{equation}

\end{thm}

Note that the LIL has been intensively investigated  for
many different models,
see e.g. \cite{BM,Chen,DL,DST,HH,KV,St} and   references therein. Theorem \ref{th3}   is a supplement in the setting of  path-dependent SDEs.

\

The remainder of  this paper is arranged as follows.  In Section 2, we recall some known results on SLLN, CLT and LIL for Markov processes, which are then applied to prove the above three results in Sections 3-5 respectively.

 \section{Some known results }

We first state some results presented in \cite{Sh} for continuous
 Markov processes on   separable Hilbert spaces. Since   proofs
of these results only use the norm rather than   the inner product
of the space, they apply also to a Banach space.

Let $\{X_t^x:\ x\in \mathbb B, t\ge 0\}$   be a continuous Markov
process on a separable Banach space $(\mathbb B,\|\cdot\|)$ with
respect to a complete filtration probability space
$(\OO,\F,(\F_t)_{t\ge 0},\P)$  such that the associate Markov
semigroup
$$P_t f(x):= \E f(X_t^x),\ \ t\ge 0, x \in \mathbb B, f \in \B_b(\mathbb B)$$ has a unique invariant probability measure $\mu_\infty$. For
a constant $\gg\in (0,1]$
 and an increasing function  $w\in C([0,\infty);[1,\infty)),$  let $C_{w,\gg}(\BB)$ be the class of measurable functions on $\BB$ such that
$$\|f\|_{w,\gg}:= \sup_{x\in \BB} \ff{|f(x)|}{w(\|x\|)}+ \sup_{x,y\in \BB} \ff{|f(x)-f(y)|}{(1\land \|x-y\|^{\gg})(1+w(\|x\|)+w(\|y\|))}<\infty.$$ Note that
 in \cite{Sh}  $\|f\|_{w,\gg}$   is defined by using $\|x-y\|^\gg$ instead of $1\land \|x-y\|^\gg$,
but this does not make essential differences since these two
definitions are equivalent up to a constant multiplication. We take
the present formulation in order to apply the ergodicity result
derived in \cite{BWY18}. By \cite[Proposition 2.6]{Sh}, we have the
following result.

\begin{lem}\label{A00}
If  there exist   $\varphi,\psi   \in C(\R_+;\R_+)$  with $\int_0^\infty \varphi(t)\d t<\infty$ such that
 \begin{equation}\label{AA}
|P_tf(x)-\mu_\infty(f)|\le
\varphi(t)\psi(\|x\|)\|f\|_{w,\gg},~~~f\in C_{w,\gg}(\BB),\ \ t\ge
0,\ \  x\in\BB,
\end{equation}
and for some $k\in\mathbb N$,
\begin{equation}\label{AA1}
\E\psi(\|X_t^x\|)^{2k}<\infty,~~~t\ge0,~x\in\BB,
\end{equation}
then for any $f\in C_{w,\gg}(\BB)$,
\begin{equation}\label{AA2}
\begin{split} &\E\Big|\ff{1}{t}\int_0^tf(X_s^x)\d
s-\mu_\infty(f)\Big|^{2\kappa}\\
&\le t^{-k} \Big(2k(2k-1)\varphi(0)\int_0^\8\varphi(s)\d s\Big)
\|f\|_{p,\gg}^{2k} \,\E\psi(\|X_t^x\|)^{2k},\ \ t\ge 1.\end{split}
\end{equation}
\end{lem}
Next,   \cite[Corollary 2.4]{Sh} gives the following result on SLLN.

\begin{lem}\label{s0}
Under   conditions of Lemma \ref{A00},  if there exist a constant
$q\in (0,1/2)$,  a function $\tau\in C(\R_+;\R_+)$  and    random
variables $\{M_x\ge 1: x\in\BB\}$ such that
\begin{equation}\label{AA4}
\E M_x^{\ff{1}{q}}\le\tau(\|x\|), \ \ x\in \BB,\end{equation}
\begin{equation}\label{AA3}
\P\Big(\|X_t^x\|\le w^{-1}(t^q)~\mbox{ for } t\ge
M_x\Big)=1,\ \ x\in\BB,
\end{equation}
 where $w^{-1}$ is the inverse of $w$. Then for any  $\vv\in(0,\ff{1}{2})$,  there exist   a constant $c_\vv>0$  and a family of random variables $\{T_{\vv,x}^f\ge 1: x\in\BB, f\in C_{w,\gg}(\BB)\}$     such that $\P$-a.s.
\begin{equation}\label{K2}
\Big|\ff{1}{t}\int_0^tf(X_s^x)\d s-\mu_\infty(f)\Big|\le
c_\vv\|f\|_{w,\gg}t^{-\ff{1}{2}+\vv},~~~t\ge T_{\vv,x}^f, \,
x\in\BB,\, f\in C_{w,\gg}(\BB),
\end{equation} and
\begin{equation}\label{K3}
\sup_{f\in C_{w,\gg}(\BB)} \ff{\E |T_{\vv,x}^f|^k}{1+\|f\|_{w,\gg}^{k(1+k)} } <\infty,\ \  k\in \mathbb N.\end{equation}

\end{lem}

Let  $f\in C_{w,\gg}(\BB)$ and $x\in\BB$,      assume that 
\beq\label{MGG} M_t^{f,x}:= \int_0^t \big\{f(X_s^x)-P_s f(x)\big\}\d s + \int_t^\infty \big\{P_{s-t} f(X_t^x)- P_s f(x)\big\}\d s,\ \ t\ge 0\end{equation} is a
 well-defined square integrable    martingale. Consider its discrete time quadratic
 variation process
$$\<M^{f,x}\>_k:= \sum_{i=1}^k \E \big((M_i^{f,x}- M_{i-1}^{f,x})^2\big|\F_{i-1}\big),\ \ k\in \mathbb N.$$
Let $\lfloor t\rfloor=\sup\{k\in\mathbb Z_+: k\le t\}$ be the integer part of $t\ge 0$.  The following  CLT is due to \cite[Theorem 2.8]{Sh}.

\begin{lem}\label{APP}  Let  $f\in C_{w,\gg}(\BB)$ and $x\in\BB$ such that $M_t^{f,x}$ in $\eqref{MGG}$ is a well-defined square integrable martingale. Assume that
\begin{equation}\label{G1}
\E\Big(\sup_{t\in[k,k+1 ]}\e^{|\psi(\|X_t^x\|)|^\aa}\Big)\le
\kk(\|x\|),~~~~k\ge0,~~x\in\BB
\end{equation} holds for some constant   $\aa>0$ and  continuous function $\kk:\R_+\to\R_+$.  Then  \begin{enumerate}
\item[$(1)$] For any constants  $D,q>0$  and $\vv\in (0,1/4)$,  there exists an increasing function $h: [0,\infty)\times[0,\infty)\to [0,\infty)$ such that for any $x\in\BB$ and $f\in C_{w,\gg}(\BB)$,
\begin{equation*}\label{K4}
\begin{split}
&\sup_{z\in\R}\Big|\P\Big(\ff{1}{\ss t}\int_0^tf(X_s^x)\d s\le
z\Big)-\Phi_{D}(z)\Big|\\
&\le
t^{-\ff{1}{4}+\vv}h(\|x\|,\|f\|_{w,\gg})
 +D^{-4q}\lfloor t\rfloor^{q(1-4\vv)}\E\big|\lfloor
t\rfloor^{-1} \<M^{f,x}\>_{\lfloor t\rfloor}-D^2\big|^{2q},~~t\ge1;
\end{split}
\end{equation*}
\item[$(2)$]  There exists an increasing function $h: [0,\infty)\times[0,\infty)\to [0,\infty)$ such that for any $x\in\BB$ and $f\in C_{w,\gg}(\BB)$,
\begin{equation*}
\begin{split}
&\sup_{z\in\R}\Big((1\wedge|z|)\Big|\P\Big(\ff{1}{\ss
t}\int_0^tf(X_s^x)\d s\le z\Big)-\Phi_0(z)\Big|\Big)\\
&\le t^{-\ff{1}{4}}h(\|x\|,\|f\|_{w,\gg})+\lfloor
t\rfloor^{-\ff{1}{2}}( \E \<M^{f,x}\>_{\lfloor
t\rfloor})^{1/2},~~~t\ge1.
\end{split}
\end{equation*}
\end{enumerate}

\end{lem}

Finally, let $(M_n)_{n\ge0}$ be a square integrable martingale  and
let
$Z_n=M_n-M_{n-1}$ be the martingale difference.
The following result is taken from \cite[Theorem 1]{HS}.

 \begin{lem}\label{L1}
  Assume  that $S_n:= \E M_n \to\8$ as $n\to\8$,  and   there exists a constant   $\dd >0$ such that
\begin{equation}\label{J1}
\sum_{n=1}^\8S_n^{-4}\E(Z_n^4{\bf1}_{\{|Z_n|\le \dd
S_n\}})<\8,~~~~~\sum_{n=1}^\8S_n^{-1}\E(Z_n{\bf1}_{\{|Z_n|\le \dd
S_n\}})<\8,
\end{equation} and $\P$-a.s.
\begin{equation}\label{J2}
 \lim_{n\to\infty} S_n^{-2}\sum_{k=1}^n Z_k^2=1.
\end{equation}
Then the sequence $(\LL_n)_{n\ge 1}$ of random variables on
$C([0,1];\R)$ defined by
\begin{equation*}
\LL_n(t)=  \sum_{k=0}^{n-1} 1_{\{S_k^2\le t S_n^2\le S_{k+1}^2\}}
\ff{M_k+(S_n^2t-S_k^2)(S_{k+1}^2-S_k^2)Z_{k+1}}{\ss{2S_n^2\log\log
S_n^2}},\ \ t\in [0,1]
\end{equation*}
  is almost surely relatively compact,
 and the set of its limits points coincides with $\mathcal{H}$ in $\eqref{*PQ}$.
\end{lem}

\section{Proof of Theorem \ref{th1}}\label{sec2}

It suffices to verify conditions in Lemmas \ref{A00} and \ref{s0}
for the present model, where $\BB=\C, w(r)= 1+r^{p/2},r\ge0$. To
this end, we present the following lemma.

\begin{lem}\label{lem1}
 Under assumptions of Theorem $\ref{th1}$,   for any $p\ge 1$ and $\gg\in (0,1]$,    there
exist constants $c,\bb>0 $ such that
\begin{equation}\label{eq5}
\E\|X_t^\xi\|^p_\8 \le c\,(1+ \e^{-\bb t}\|\xi\|_\8^p),~~~\xi\in\C, t\ge 0,
\end{equation} and
 \begin{equation}\label{eq6}
\mathbb {W}_{p, \gg}(\mu P_t, \nu P_t)\le c\,\e^{-\bb
t}\mathbb {W}_{p,\gg}(\mu, \nu),\ \ \mu,\nu  \in \scr
P_{p, \gg}(\C), t\ge 0.
\end{equation} Consequently, $P_t$ has a unique invariant probability measure $\mu_\infty$ and  $\mu_\infty(\|\cdot\|_\8^p)<\infty$ for all $p\ge 1.$
 \end{lem}

\begin{proof} (1)    By Jensen's inequality, concerning  \eqref{eq5} we only need to consider $p\ge 2$.
Since  $ \ll_1-\ll_2\e^{\ll_1r_0}>0$, there exists a constant   $\vv
\in (0,\ll_1)$
 such that
\begin{equation}\label{eq0}
\ll_\vv:=\ll_1-\ll_2\e^{(\ll_1-\vv)r_0}-\vv>0.
\end{equation}
According to  ({\bf A1}) and ({\bf A3}), we may find   a constant
$c_0>0$ such that
\begin{equation*}
2\,\<\xi(0),b(\xi)\>+\|\si(\xi)\|_{\rm HS}^2\le
c_0-(\ll_1-\vv)|\xi(0)|^2+\ll_2\|\xi\|_\8^2,~~~\xi\in\C.
\end{equation*}
  So,   by It\^o's formula,
\begin{equation}\label{eq2}
\begin{split}
\e^{(\ll_1-\vv)
t}|X^\xi(t)|^2&=|\xi(0)|^2+M^\xi(t)+\int_0^t\e^{(\ll_1-\vv)
s}\Big((\ll_1-\vv)|X^\xi(s)|^2\\
& \quad+2\<X^\xi(s),b(X_s^\xi)\>+\|\si(X_s^\xi)\|_{\rm HS}^2 \Big)\d s \\
&\le |\xi(0)|^2+M^\xi(t)+c_1\,\e^{(\ll_1-\vv)
t}+\ll_2\int_0^t\e^{(\ll_1-\vv) s}\|X_s^\xi\|_\8^2\d s
\end{split}
\end{equation} holds for some constant $c_1>0$ and the martingale
$$M^\xi(t):=2\int_0^t \e^{(\ll_1-\vv)  s}\<\si^*(X_s^\xi)X^\xi(s),\d W(s)\>,~~~t\ge0.$$ Noting that
\begin{equation*}
\e^{(\ll_1-\vv) t}\|X_t^\xi\|_\8^2\le \e^{(\ll_1-\vv) r_0
}\Big(\|\xi\|_\8^2\vee \sup_{0\le s\le t}(\e^{(\ll_1-\vv)
s}|X^\xi(s)|^2)\Big),
\end{equation*} we deduce from \eqref{eq2} that
\begin{equation*}
\begin{split}
\|X_t^\xi\|_\8^2&\le \e^{(\ll_1-\vv) r_0 }\Big\{c_1+
\e^{-(\ll_1-\vv) t}\|\xi\|_\8^2+\e^{-(\ll_1-\vv)
t}N^\xi(t)\\
&\quad+\ll_2\int_0^t\e^{-(\ll_1-\vv) (t-s)}\|X_s^\xi\|_\8^2\d
s\Big\},\ \ t\ge 0,
\end{split}
\end{equation*}
where  $N^\xi(t):=\sup_{0\le s\le t}M^\xi(s)$. By invoking
Gronwall's inequality (see e.g. \cite[Theorem 11]{Dra}), this
implies
\begin{equation*}%\label{X1}
\begin{split}
\|X^\xi_t\|^2_\8&\le\e^{(\ll_1-\vv) r_0 }\Big\{c_1 +\e^{-(\ll_1-\vv) t}\|\xi\|_\8^2+ \e^{-(\ll_1-\vv)
t}N^\xi(t)\Big\}\\
&\quad+\ll_2\e^{2(\ll_1-\vv) r_0 } \int_0^t\Big\{c_1+
\e^{-(\ll_1-\vv) s}\|\xi\|_\8^2+ \e^{-(\ll_1-\vv)
s}N^\xi(s)\Big\}\e^{-\ll_\vv(t- s)}\d s,\ \ t\ge 0.
\end{split}
\end{equation*} Combining this with H\"older's  inequality,   for fixed  $p\ge 2$  we may find   constants $c_2,c_3>0$ such that
\begin{equation}\label{A0}
\begin{split}
 & \E\|X^\xi_t\|^p_\8\le c_2 + c_2  \e^{-\ff{p}{2}\ll_\vv
 t}\|\xi\|_\8^p+ c_2 \e^{-\ff{p}{2} (\ll_1-\vv)
 t}(N^\xi(t))^{p/2} \\
 &\qquad  \qquad +c_2 \E \bigg|\int_0^t\e^{-(\ll_1-\vv)
s}\e^{-\ll_\vv(t- s)}N^\xi(s)\d s\bigg|^{p/2} \\
&\le c_3   +c_3 \e^{-\ff{p}{2}\ll_\vv
 t}\|\xi\|_\8^p+c_3 \e^{-\ff{p}{2} (\ll_1-\vv)
 t}\E (N^\xi(t))^{p/2} \\
 &\qquad\qquad  +c_3  \int_0^t\e^{-\ff{p}{2} (\ll_1-\vv)
 s-\ll_\vv(t-s)}\E(N^\xi(s))^{p/2} \d s.
\end{split}
\end{equation}
On the other hand, by means of {\bf (A3)} and using   BDG's and
H\"older's inequalities, there exist constants $c_4,c_5>0$ such that
\begin{equation*}
\begin{split}
& \e^{-\ff{p}{2}(\ll_1-\vv)
 t}\E(N^\xi(t))^{p/2}\le   c_4\E\Big(\int_0^t\e^{-2(\ll_1-\vv)(t-s)}|X^\xi(s)|^2\d
 s\Big)^{p/4}\\
 &\le c_4 \E\bigg[\bigg(\int_0^t \e^{-2(\ll_1-\vv)(t-s)} |X^\xi(s)|^p \d s \bigg)^{\ff 1 2} \bigg(\int_0^t \e^{-2(\ll_1-\vv)(t-s)}\d s\bigg)^{\ff{p-2}4}\bigg] \\
 &\le c_5 + \ff{(1\land\ll_\vv)^2}{4 c_3}   \int_0^t \e^{-2(\ll_1-\vv)(t-s)}\E  |X^\xi(s)|^p \d s,\ \ t\ge 0.\end{split}
\end{equation*}
Substituting this into \eqref{A0}, and noting that  due to $\ll_1-\vv>\ll_\vv>0$ we have
\beg{align*}& \int_0^t\e^{-\ll_\vv(t-s)} \d s  \int_0^s \e^{-2(\ll_1-\vv)(s-r)}\E  |X^\xi(r)|^p \d r \\
&=\int_0^t \e^{2(\ll_1-\vv)r-\ll_\vv t} \E|X^\xi(r)|^p\d r \int_r^t \e^{-(2(\ll_1-\vv)-\ll_\vv)s}\d s\\
& \le \ff 1 {2(\ll_1-\vv)-\ll_\vv} \int_0^t \e^{ -\ll_\vv (t-r)} \E|X^\xi(r)|^p\d r\\
&\le \ff 1 {\ll_\vv} \int_0^t \e^{-\ll_\vv(t-r)} \E|X^\xi(r)|^p\d r,\end{align*}
we   may find  a constant $C>0$ such that
$$ \E\|X^\xi_t\|^p_\8\le  C(1+\|\xi\|_\8^p)+ \ff {\ll_\vv} 2 \int_0^t\e^{-\ll_\vv(t-s)} \E \|X_s^\xi\|_\8^p\d s,\ \ t\ge 0.$$ By a truncation argument with stopping times, we may and do assume that $\E\|X^\xi_t\|^p_\8<\infty$, so that by Gronwall's inequality, this implies the desired
estimate \eqref{eq5} for some constants $c,\bb>0.$

  (b) By \eqref{eq5},
 the Lyapunov condition  {\bf (A3)} in \cite[Theorem 1.1]{BWY18}
holds for $V(\xi):=\|\xi\|_\8^p,\xi\in\C$  and $\gamma= \bb.$ In
terms of \cite[Theorem 1.1]{BWY18}, this together with {\bf (A1)}
and {\bf (A2)}   implies
  \eqref{eq6} for possibly different constants $c,\bb>0$, which then implies the existence and uniqueness of the invariant probability measure $\mu_\infty \in  \scr P_{p,\gg}(\C)$. Since $p\ge 1$ is arbitrary, we conclude that $\mu_\infty(\|\cdot\|_\8^p)<\infty$ holds for all $p\ge 1.$
\end{proof}

\begin{proof}[Proof of Theorem \ref{th1}] From \eqref{ABC} and        \eqref{eq5}   we see that assumptions in Lemma \ref{A00} holds for $\BB= \C, w(r)=1+r^{p/2},
k=1, \varphi(t)=c\,\e^{-\bb t}, $ and
  $\psi(r)=1+r^{p/2}.$
Then  (1) follows from Lemma \ref{A00} .

Next, to prove (2),  we only need to verify conditions \eqref{AA4} and \eqref{AA3} in Lemma \ref{s0}.  For  $q\in(0,1/2)$, consider the following $[0,\infty]$-valued random variables:
\begin{align*}
&M:=\inf\Big\{T\ge0: 16^{\ff{1}{p}} \|X_t^\xi\|^2_\8\le t^{\ff{4q}{p}}
~\mbox{ for }~t\ge T\Big\},\\
 & M':=\inf\Big\{m\in\mathbb{N}: 16^{\ff{1}{p}} \sup_{t\in[k,k+1]}\|X_t^\xi\|_\8^2\le k^{\ff{4q}{p}}  ~\mbox{ for
}~ \mathbb{N}\ni k \ge m+1 \Big\}.
\end{align*} Obviously, $M\le M'$.
 Since \begin{equation} \label{e1}
\sup_{t\in[k,k+1]}\|X_t^\xi\|_\8\le
\max_{i\in\{0,1,\cdots,\lfloor1/r_0\rfloor+1\}}\|X_{k+ir_0}^\xi\|_\8,
\end{equation}
by   \eqref{eq5} and applying
Chebyshev's inequality, we may find   a   constant $C (\xi)>0$ such that
\begin{align*}
&\sum_{k=1}^\8\P\Big(\sup_{t\in[k,k+1]}\|X_t^\xi\|_\8^2 \ge\ff{k^{\ff{4q}{p}}}{16^{\ff{1}{p}}}\Big)\\
  &\le  2^{1+\ff{1}{q}}\sum_{k=1}^\8
\ff{\E(\sup_{t\in[k,k+1]}\|X_t^\xi\|_\8^{\ff{p}{2}(1+\ff{1}{q})})}{k^{1+q}}
 \le C (\xi)\sum_{k=1}^\8 \ff{1}{k^{1+q}}<\8.
\end{align*}
So, by   Borel-Cantelli's lemma,
  there exists an $\mathbb N$-valued  random variable $K$ such that
\begin{equation*}
\P\Big(\sup_{t\in[k,k+1]}\|X_t^\xi\|_\8^2\le\ff{k^{\ff{4q}{p}}}{16^{\ff{1}{p}}}
~\mbox{ for }~k\ge K\Big)=1.
\end{equation*}
Therefore, $\P$-a.s.  $M \le M'<\infty$ and  \eqref{AA3} holds true.
Moreover,   \eqref{eq5} and  Chebyshev's inequality also imply
\begin{equation*}
\begin{split}
\E\,|M'|^{\ff{1}{q}}&=\sum_{k=0}^\8k^{\ff{1}{q}}\P(M'=k)\le
\sum_{k=1}^\8k^{\ff{1}{q}}\P\Big(\sup_{t\in[k,k+1]}\|X_t^\xi\|_\8^2>\ff{k^{\ff{4q}{p}}}{16^{\ff{1}{p}}}\Big)\\
&\le16^{\ff{\aa}{p}}\sum_{k=1}^\8\ff{\E(\sup_{t\in[k,k+1]}\|X_t\|_\8^{2\aa})}{k^{1+q}
}\\
&\le  c\, (1+\|\xi\|_\8^{2\aa}),~~~ \aa:=\ff{p}{4\,q}(1+q+1/q)
\end{split}
\end{equation*}
for some constant $c>0.$ This,
 together with $M\le M'$, leads to
\begin{equation}\label{e6}
\E\,M^{\ff{1}{q}}\le c\,(1+\|\xi\|_\8^{2\aa}),~~~q\in(0,1/2),
\end{equation}
which ensures condition \eqref{AA4}. Therefore, the proof is finished by  Lemma \ref{s0}.
\end{proof}

\section{Proof of Theorem \ref{th2}}\label{sec3}

To apply Lemma \ref{APP},  for fixed $f\in C_{p,\gg}(\C) $ with
$\mu_\infty(f)=0$,    consider \beq\label{MTG} M_t^{f,\xi} :=
\int_0^t \big\{f(X_u^\xi)- P_u f(\xi)\big\}\d u + \int_t^\infty
\big\{P_{u-t} f(X_t^\xi)- P_u f(\xi)\big\}\d u,\ \ t\ge 0,\, \xi\in
\C.\end{equation} Since $\mu_\infty(f)=0$, \eqref{ABC} implies
$$ |P_t f(\xi)|\le c\,\e^{-\bb t} \|f\|_{p,\gg}(1+\|\xi\|_\8^{p/2}),\ \ t\ge 0, \,\xi\in \C$$
for some constants $c,\bb>0$. So,  there exists an increasing function   $c:\R_+\to\R_+$ such that \eqref{eq5} yields
\beg{align*} & \E \bigg|\int_0^t \big\{f(X_u^\xi)- P_u f(\xi)\big\}\d u+ \int_t^\infty \big|P_{u-t} f(X_t^\xi)- P_u f(\xi)\big|\d u\bigg|^k \\
&\le  c(t) \|f\|_{p,\gg}^k \bigg( 1+  \E\|X_t^\xi\|_\8^{pk/2} +\|\xi\|_\8^{pk/2}+\int_0^t(1+\|\xi\|_\8^{pk/2}+ \E\|X_u^\xi\|_\8^{pk/2}) \d u %\int_t^\infty \e^{-\bb (u-t)} \big(
\bigg)\\
&<\infty,\ \ t\ge 0. \end{align*} Hence,     $M_t^{f,\xi}$ is a
well-defined  martingale with $\E|M_t^{f,\xi}|^k<\infty$ for all
$k\ge 1$.

Next,  consider
 \begin{equation*}\label{A1}
\<M^{f,\xi}\>_k:=\sum_{i=1}^k\E\Big((M_i^{f,\xi}-M^{f,\xi}_{i-1})^2\Big|\F_{i-1}\Big),~~~k\in\mathbb N.
\end{equation*}
Let $R_f$ and $\varphi_f$ be  defined as  in \eqref{WWW} and
\eqref{S7}, respectively.  By the Markov property of
$(X_t^\xi)_{t\ge0}$, we have
$$
M_i^{f,\xi} =M_{i-1}^{f,\xi}+\int_{i-1}^i f(X_u^\xi) \d
u+R_f(X_i^\xi)-R_f(X_{i-1}^\xi)
$$
so that
 $$
\E((M_i^{f,\xi}-M_{i-1}^{f,\xi})^2|\F_{i-1})
=\varphi_f(X_{i-1}^\xi).  $$ Consequently, we arrive at
\begin{equation}\label{D6}
\<M^{f,\xi}\>_k=\sum_{i=0}^{k-1}\varphi_f(X_i^\xi),\ \ k\in\N.
\end{equation}

\begin{lem}\label{Lem3}
 Under assumptions of Theorem \ref{th1}, for any  $f\in
C_{p,\gg}(\C) $ with $\mu_\infty(f)=0$,
\begin{equation}\label{D5}
0\le  \mu_\infty(\varphi_f)=2\,\mu_\infty(fR_f)<\8.
\end{equation}
  \end{lem}

\begin{proof} Firstly, by Lemma \ref{lem1} and \eqref{ABC}, we have $\mu_\infty(\|\cdot\|_\8^p)<\infty$ for all $p\ge 1$ so that
$\mu_\infty(\varphi_f)<\infty$ for any $f\in C_{p,\gg}(\C)$ with $\mu_\infty(f)=0$.

Next, by the Markov property of $(X_t^\xi)_{t\ge0}$ and noting that  \eqref{WWW} implies
\begin{equation*}
P_tR_f(\xi)=R_f(\xi)-\int_0^tP_sf(\xi)\d s,~ t\ge0,
\end{equation*}
we have
$$ \E\big[f(X_s^\xi) R_f(X_1^\xi)\big] = P_s(fP_{1-s}R_f)(\xi) = P_s (fR_f)(\xi) -\int_0^{1-s} P_s (f P_r f)(\xi)\d r,\ \ s\in [0,1],$$ and
\beg{align*}& \E\bigg(\int_0^1 f(X_s^\xi)\d s\bigg)^2= 2\E\int_0^1 f(X_s^\xi)\d s \int_s^1 f(X_r^\xi)\d r\\
&=2\int_0^1\d s\int_s^{1} P_s(fP_{r-s}f)(\xi) \d r = 2\int_0^1\d s\int_0^{1-s} P_s(fP_{r}f)(\xi) \d r.\end{align*}
Then it follows from \eqref{S7} that
\begin{equation}\label{VPH}
\begin{split}
\varphi_f(\xi)&=R_f( \xi)^2+P_1(R_f)^2 (\xi) + \E\bigg(\int_0^1 f(X_s^\xi)\d s\bigg)^2+ 2 \int_0^1 \E\big[f(X_s^\xi) R_f(X_1^\xi)\big] \d s \\
&\quad - 2 R_f(\xi)\int_0^1 P_rf(\xi)\d r -2 R_f(\xi) P_1 R_f(\xi)\\
&= R_f( \xi)^2+P_1(R_f)^2 (\xi)   + 2\int_0^1\d s\int_0^{1-s}P_s(fP_{r}f)(\xi) \d r + 2\int_0^1 P_s (fR_f)(\xi) \d s \\
&\quad - 2 \int_0^1\d s \int_0^{1-s} P_s(fP_r f)(\xi) \d r- 2 R_f(\xi) \int_0^1 P_r f(\xi) \d r - 2 R_f(\xi)^2 \\
&\qquad\qquad + 2 R_f(\xi) \int_0^1 P_s f(\xi)\d s\\
&= P_1 (R_f^2)(\xi) - R_f(\xi)^2 +   2\int_0^1 P_s (fR_f)(\xi) \d s.
\end{split}
\end{equation}
Since $\mu_\infty$ is $P_t$-invariant,   integrating with respect to
$\mu_\infty(\d\xi)$ on both sides of \eqref{VPH} gives
$\mu_\infty(\varphi_f)= 2\mu_\infty(fR_f)$.

\end{proof}

\begin{lem}\label{Lem2}
 Under      assumptions of Theorem $\ref{th1}$, there exists a constant $C>0$ such that
\begin{equation}\label{W2}  \|\varphi_f\|_{2p,\gg}\le C\,\,\|f\|_{p,\gg}^2,\
\ \ f\in C_{p,\gg}(\C), \mu_\infty(f)=0.
\end{equation}

\end{lem}

\begin{proof}  By \eqref{WWW} and \eqref{ABC}, in addition to $\mu_\infty(\|\cdot\|_\8^p)<\infty$, there exists a constant $c_1>0$ such that
\beq\label{PP1} |R_f(\xi)|\le c_1 \|f\|_{p,\gg} (1+\|\xi\|_\8^p),\ \
f\in C_{p,\gg}(\C), \xi\in \C.\end{equation} Next, applying
\eqref{eq6} to $\mu=\dd_\xi$ and $\nu=\dd_\eta$, we obtain
\begin{equation}\label{d3}
|P_tf(\xi)-P_tf(\eta)|\le c\,\e^{-\bb
t}\|f\|_{p,\gg}\rho_{p,\gg}(\xi,\eta).
\end{equation}
This and \eqref{WWW} imply
\begin{equation}\label{u1}
|R_f(\xi)-R_f(\eta)|\le \int_0^\8|P_tf(\xi)-P_tf(\eta)|\d t\le
\ff{c}{\bb}\|f\|_{p,\gg}\rho_{p,\gg}(\xi,\eta).
\end{equation}
Moreover, it follows from \eqref{PP1} and \eqref{u1} that
\begin{equation}\label{A3}
|R_f(\xi)^2-R_f(\eta)^2|=|R_f(\xi)+R_f(\eta)|\cdot|R_f(\xi)-R_f(\eta)|
 \le c' \,\|f\|_{ p,\gg}^2\,\rho_{2p,\gg}(\xi,\eta)
\end{equation}
for some constant  $c'>0$.  Combining   \eqref{d3}-\eqref{A3} with \eqref{VPH}, we finish the proof. \end{proof}

\begin{lem}\label{lem3}
 Under assumptions of Theorem \ref{th1}, there exist constants
$\dd, c>0$ such that
\begin{equation}\label{Y1}
\E\Big(\sup_{t\in[k,k+1]}\,\e^{\dd\,\|X_t^\xi\|_\8^2}\Big)\le\e^{c\,(1+\|\xi\|_\8^2)},~~k\ge0,~\xi\in\C,~~t\ge0.
\end{equation}
 \end{lem}

\begin{proof}
In terms of \cite[Lemma 2.1]{BWY}, there exist constants
$c_0,\vv_0>0$ such that
\begin{equation}\label{E1}
\sup_{t\ge0}\E\e^{\vv_0\|X_t^\xi\|_\8^2}\le
\e^{c_0(1+\|\xi\|_\8^2)},~~~\xi\in\C.
\end{equation}
On the other hand,  \eqref{e1} implies
\begin{equation*}
\begin{split}
\E\Big(\sup_{t\in[k,k+1]}\,\e^{\vv_0\,\|X_t^\xi\|_\8^2}\Big)
&\le\E\Big(
\max_{i\in\{0,1,\cdots,\lfloor1/r_0\rfloor+1\}}\e^{\vv_0\|X_{k+ir_0}^\xi\|_\8^2}\Big)\\
&\le(\lfloor1/r_0\rfloor+2)\max_{i\in\{0,1,\cdots,\lfloor1/r_0\rfloor+1\}}\E\,
\e^{\vv_0\|X_{k+ir_0}^\xi\|_\8^2}.
\end{split}
\end{equation*}
Combining this with \eqref{E1}, we prove  \eqref{Y1}.
\end{proof}

\begin{proof}[Proof of Theorem \ref{th2}] Let $f\in C_{p,\gg}(\C)$ with $\mu_\infty(f)=0$. By Lemmas \ref{lem1} and \ref{lem3}, the results in Lemma \ref{APP} applies to $D=D_f.$ Below we consider
$D_f>0$ and $D_f=0$, respectively.

(a) Let $D_f>0$. By Lemma \ref{APP}(1), for any $\vv, q>0$, there
exists an increasing function $h:\R_+\times\R_+\to\R_+$ such that
\begin{equation}\label{F4}
\begin{split}
& \sup_{z\in\R}\Big|\P\Big(\ss t A_t^f (\xi)   \le
z\Big)-\Phi_{D_f}(z)\Big|\\
&\le
  h_1(\|\xi\|_\8,\|f\|_{p,\gg})t^{-\ff{1}{4}+\vv}
 +D_f^{-4q}\lfloor t\rfloor^{q(1-4\vv)}\E\Big|\ff{1}{\lfloor
t\rfloor}\<M^{f,\xi}\>_{\lfloor t\rfloor}-D_f^2\Big|^{2q},\ \ t\ge 1.
\end{split}
\end{equation}
So, if we can find   an increasing function $\hat h:\R_+\times\R_+
\to\R_+$ such that
\begin{equation}\label{F3}
\E\Big|\ff{1}{\lfloor t\rfloor}\<M^{f,\xi}\>_{\lfloor
t\rfloor}-D_f^2\Big|^{2q} \le \hat h( \|\xi\|_\8,\|f\|_{p,\gg} )
\lfloor t\rfloor^{-q} , ~\xi\in\C,~~t\ge1,
\end{equation} then the desired estimate   in Theorem \ref{th2}(1) follows from \eqref{F3} with   large enough $q>0$, say, $q>\ff 1 {16\vv}.$
 By \eqref{ABC} for $2p$ instead of $p$,
 $$
|P_t\varphi_f(\xi)-D_f^2|  \le  c\,\|\varphi_f\|_{2p,\gg}\,\e^{-\bb
t} (1+\|\xi\|^p_\8)
 $$ holds for some constants $c,\bb>0$. Combining this with \eqref{eq5}, \eqref{D6} and \eqref{W2}, we prove \eqref{F3}.

(b) Let $D_f=0$. With  $q=1$ the estimate  \eqref{F3} reduces to
\begin{equation}\label{F6} \E\Big|\ff{1}{\lfloor
t\rfloor}\<M^{f,\xi}\>_{\lfloor t\rfloor} \Big|^2 \le \hat h (
\|\xi\|_\8,\|f\|_{p,\gg}) \lfloor t\rfloor^{-1} , ~\xi\in\C,~~t\ge1.
\end{equation}
 Combining this with Lemma \ref{APP}(2), we prove Theorem \ref{th2}(2).
\end{proof}

\section{Proof of Theorem \ref{th3}}\label{sec4}

 Let us fix  $  f\in C_{p,\gg}(\C) $ with $\mu_\infty(f)=0$. To apply Lemma \ref{L1}, for any $\xi\in\C$, we consider
$$ M_n^{\xi}:= \sum_{k=0}^n  \big\{f(X_k^\xi)-P_kf(\xi)\big\} +  \sum_{k=n+1}^\8 \big\{P_{k-n} f(X_n^\xi)-P_kf(\xi)\big\},\ \ n\ge 0.
 $$ The argument after \eqref{MTG}  implies that $(M_n^\xi)_{n\ge 0}$ is a  well-defined square integrable martingale.
Let
\begin{equation*}
 ~S_n^{\xi} =\ss{\E\big| M_n^{\xi}\big|^2}\,,~~
  Z_n^{\xi} =  M_n^{\xi}- M_{n-1}^{\xi},\ \ n\ge 1,
\end{equation*}
and let $\hat R_f$ and $\hat\varphi_f$ be given before Theorem
\ref{th3}. Following the arguments of   Lemmas \ref{Lem3} and
  \ref{Lem2}, we have
\begin{equation}\label{w1}
0\le \hat D_f^2:=\mu_\infty(\hat\varphi_f)=2\mu_\infty(f\hat R _f)<\8,
\end{equation}
and for some constant $ c>0,$
\begin{equation}\label{w2}
 \|\hat\varphi_f\|_{2p,\gg}\le  c\,\|f\|_{p,\gg}^2,\ \ f\in C_{p,\gg}(\C).
\end{equation}

\begin{lem}
  Under assumptions of Theorem $\ref{th1}$, $\P$-a.s.
\begin{equation}\label{w4}
\ff{1}{n}\sum_{k=1}^n(Z_k^{\xi})^2\to\hat D_f^2.
\end{equation}

\end{lem}

\begin{proof} According to  the proof of \cite[Lemma 3.2]{BM},
it suffices to show that the maps
\begin{align*}
&\C\ni \xi\mapsto \Lambda_1(\xi):=\E\Big(\Big|\limsup_{n\to\8}\Big(\ff{1}{n}\sum_{k=1}^n(Z_{k}^{\xi})^2\Big)-\hat D_f^2\Big|\wedge1\Big)\\
&\C\ni\xi\mapsto
\Lambda_2(\xi):=\E\Big(\Big|\liminf_{n\to\8}\Big(\ff{1}{n}\sum_{k=1}^n(Z_{k}^{\xi})^2\Big)-\hat
D_f^2\Big|\wedge1\Big)
\end{align*}are continuous. For simplicity, we only prove the continuity of $ \Lambda_1$ as that of the other is completely similar.
By   definition it is easy to see that
\beq\label{DD0} Z_n^\xi =f(X_n^\xi)+\sum_{k=n}^\infty \big\{P_{k-n} f(X_n^\xi)- P_{k+1-n}f(X_{n-1}^\xi)\big\},\ \ n\ge 1.\end{equation}
Combining this with \eqref{ABC}, we find  constants $c_1,c_2>0$ such that
\beq\label{DD1} \beg{split} &\big|Z_n^\xi-Z_n^\eta\big| \le \big|f(X_n^\xi)- f(X_n^\eta)\big|\\
&\quad  +\sum_{k=n}^\infty \big(|P_{k-n} f(X_n^\xi)- P_{k-n} f(X_n^\eta)| + |P_{k+1-n} f(X_{n-1}^\xi)- P_{k+1-n} f(X_{n-1}^\eta)|\big)\\
&\le c_1 \rr_{p,\gg}(X_n^\xi, X_n^\eta) + c_1 \sum_{k=n}^\infty \e^{-\bb (k-n)} \big\{\rr_{p,\gg}(X_n^\xi,X_n^\eta)+  \rr_{p,\gg}(X_{n-1}^\xi,X_{n-1}^\eta)\big\}\\
&\le c_2 \big\{\rr_{p,\gg}(X_n^\xi,X_n^\eta)+
\rr_{p,\gg}(X_{n-1}^\xi,X_{n-1}^\eta)\big\}.\end{split}\end{equation}
Similarly, \eqref{ABC} with $\mu_\infty(f)=0$ and \eqref{DD0}  also
imply
$$|Z_n^\xi|\le c_3 (1+ \|X_n^\xi\|_\8+  \|X_{n-1}^\xi\|_\8)^{p/2},\ \ n\ge 1, \xi\in \C$$
for some constant $c_3>0$. Combining this with \eqref{DD1} and
setting
$$A_k^{\xi,\eta}:= 1+\|X_k^\xi\|_\8+ \|X_k^\eta\|_\8+\|X_{k-1}^\xi\|_\8+ \|X_{k-1}^\eta\|_\8,\ \ k\ge 1,$$   we may find a constant $c_4>0$ such that
\beq\label{PP0} \beg{split} & |\LL_1(\xi)-\LL_1(\eta)|\\
&\le  \bigg| \E \Big|\lim_{l\to\infty}\sup_{n\ge l}  \ff 1 n \sum_{k=1}^n \big(|Z_k^\xi|^2 -\hat D_f^2\big)\Big|   - \E \Big|\lim_{l\to\infty}\sup_{n\ge l} \ff 1 n \sum_{k=1}^n \big(|Z_k^\eta|^2 -\hat D_f^2\big)\Big| \bigg| \\
&\le \E\bigg[\lim_{l\to\infty} \sup_{n\ge l} \ff 1 n \sum_{k=1}^n \big|Z_k^\xi-Z_k^\eta\big|  \big(|Z_k^\xi|+|Z_k^\eta|\big)\bigg]\\
&\le c_4 \E\bigg[\lim_{l\to\infty} \sup_{n\ge l} \ff 1 n \sum_{k=1}^n\big\{\rr_{p,\gg}(X_k^\xi,X_k^\eta)+ \rr_{p,\gg}(X_{k-1}^\xi,X_{k-1}^\eta)\big\}
 |A_k^{\xi,\eta}|^{\ff p 2}\bigg].\end{split}\end{equation}
Since   $\rr_{p,\gg}(\xi,\eta)\le (1+\|\xi\|_\8+\|\eta\|_\8)^{\ff p2}$ for all $\xi,\eta\in\C$, for any $m\ge 1$ and $  l\ge m$ we have
\beg{align*}&  \sup_{n\ge l} \ff 1 n \sum_{k=1}^n\big\{\rr_{p,\gg}(X_k^\xi,X_k^\eta)+ \rr_{p,\gg}(X_{k-1}^\xi,X_{k-1}^\eta)\big\} |A_k^{\xi,\eta}|^{\ff p 2}\\
&\le \ff 1 l \sum_{k=1}^m  \big\{\rr_{p,\gg}(X_k^\xi,X_k^\eta)+ \rr_{p,\gg}(X_{k-1}^\xi,X_{k-1}^\eta)\big\} |A_k^{\xi,\eta}|^{\ff p 2} \\
&\quad   + \sum_{k=m}^\infty|A_k^{\xi,\eta}|^{\ff {3p}4} \ss{ \rr_{p,\gg}(X_k^\xi,X_k^\eta)+ \rr_{p,\gg}(X_{k-1}^\xi,X_{k-1}^\eta) }.\end{align*}
  Combining this with \eqref{PP0},  \eqref{eq5}, \eqref{eq6}, and applying the Schwarz inequality, we may find constants $c_5,c_6>0$ such that
\beg{align*} &\limsup_{\eta\to\xi} |\LL_1(\xi)-\LL_1(\eta)| \le c_4 \sum_{k=m}^\infty\E\bigg[|A_k^{\xi,\eta}|^{\ff {3p}4} \ss{ \rr_{p,\gg}(X_k^\xi,X_k^\eta)+ \rr_{p,\gg}(X_{k-1}^\xi,X_{k-1}^\eta) }\bigg]\\
&\le c_5 \sum_{k=m}^\infty\big(\E\{\rr_{p,\gg}(X_k^\xi,X_k^\eta)+ \rr_{p,\gg}(X_{k-1}^\xi,X_{k-1}^\eta)\}\big)^{\ff 1 2} \big(\E|Z_k^{\xi,\eta}|^{\ff {3 p}2}\big)^{\ff 1 2} \\
&\le c_6 (1+\|\xi\|_\8+\|\eta\|_\8)^{\ff {3p}4} \sum_{k=m}^\infty
\e^{-\bb k/2},\ \ k\ge 1.\end{align*} Letting $k\to\infty$, we
consequently prove $\limsup_{\eta\to\xi}
|\LL_1(\xi)-\LL_1(\eta)|=0.$\end{proof}

\begin{proof}[{\bf Proof of Theorem \ref{th3}}]  Let  $f\in
C_{p,\gg}(\C) $   with $\mu_\infty(f)=0$ and $\hat D_f>0$,   and let
$\xi\in \C$. Below we prove assertions (1) and (2), respectively.

(1)   By Lemma \ref{L1}, for the first assertion we only need to verify conditions \eqref{J1} and \eqref{J2} for $(S_n,Z_n)= (S_n^\xi, Z_n^\xi)$.

Firstly, by \eqref{ABC} and \eqref{w2}, there exist constants $c=c(f,\xi)$ and $\bb>0$ such that
\begin{equation}\label{w3}|P_k\hat \varphi_f(\xi)-\hat D_f^2| =
|P_k\hat \varphi_f(\xi)-\mu_\infty(\hat\varphi_f)|  \le c\,\e^{-\bb
k},\ \ k\ge 0.
\end{equation}
Consequently,
\begin{equation}\label{W0}
\lim_{n\to\8}\ff{(S_n^{\xi})^2}{n}=\lim_{n\to\8}\ff{1}{n}\sum_{k=0}^{n-1}P_k\hat\varphi_f(\xi)=\hat D^2_f>0,
\end{equation}
so that  $S_n^{\xi}\to\8$ as $n\to\8.$
 Next, by following the argument to derive \eqref{u1}, there exists a constant $c_1=c_1(f)>0$ such that
 $$|\hat R_f(\xi_1)- \hat R_f(\xi_2)|\le c_1\rr_{p,\gg}(\xi_1,\xi_2),\ \ \xi_1,\xi_2\in\C.$$
 Combining this with \eqref{eq5}, we may find constants $c_2=c_2(f),c_3=c_3(f,\xi)>0$ such that
  \begin{equation*}
\begin{split}
\E|Z_n^{\xi}|^4&\le8\,  \E|f(X_{n-1}^\xi)|^4  +8\,\E|\hat R_f(X_n^\xi)-\hat R_f(X_{n-1}^\xi)|^4\\
&\le  c_2\Big\{1+ \E \|X_n^\xi\|_\8^{2p}+\E\|X_{n-1}^\xi\|_\8^{2p} \Big\}\\
&\le c_3,\ \ n\ge 1.
\end{split}
\end{equation*}  This together with \eqref{W0} yields
\begin{equation}\label{Y22}
\sum_{n=1}^\8(S_n^{\xi})^{-4}\E\Big((Z_n^{\xi})^4{\bf1}_{\{|Z_n^{\xi}|<
S_n^{\xi}\}}\Big)\le\sum_{n=1}^\8(S_n^{\xi})^{-4}\E(Z_n^{\xi})^4<\infty.
\end{equation}
Combining this with  Chebyshev's
inequality, we obtain
\begin{equation}\label{Y3}
\sum_{n=1}^\8(S_n^{\xi})^{-1}\E\Big(|Z_n^{\xi}|{\bf1}_{\{|Z_n^{\xi}|\ge
S_n^{\xi}\}}\Big)\le
\sum_{n=1}^\8(S_n^{\xi})^{-4}\E(Z_n^{\xi})^4 <\8.
\end{equation}
Therefore, \eqref{J1} holds true for $(S_n,Z_n)= (S_n^{\xi},Z_n^\xi).$

On the other hand,     \eqref{w4} and \eqref{W0} imply $\P$-a.s.
\begin{equation}\label{Y33}
\lim_{n\to\infty} \ff{1}{(S_n^{\xi})^2}\sum_{k=1}^n(Z_k^{\xi})^2=\lim_{n\to\infty} \ff{n}{(S_n^{\xi})^2}\bigg(\ff{1}{n}\sum_{k=1}^n(Z_k^{\xi})^2\bigg)=1.
~~\P\mbox{-a.s.}
\end{equation} So,  \eqref{J2} holds for  $(S_n,Z_n)= (S_n^{\xi},Z_n^\xi)$ as well, and hence the assertion in (1) follows from Lemma \ref{L1}.

(2) It remains to prove \eqref{B2}. By the first assertion,
$\Lambda_n^{f,\xi}(t)$    is almost surely relatively compact in
$C([0,1];\R)$ and the set of its limits points coincides with
$\mathcal {H}$. Since  $\|h\|_\mathcal {H}\le 1$ for any
$h\in\mathcal {H}$, this implies $\P$-a.s.
\begin{equation}\label{B3}
\limsup_{n\to\8}\sup_{t\in[0,1]}|\Lambda_n^{f,\xi}(t)|\le1.
\end{equation}
Observing that \eqref{B1} implies \beq\label{*WP}\LL_n^{f,\xi}(1)=
\ff{\sum_{l=1}^{n-1} f(X_l^\xi)}{\hat D_f \ss{2n\log\log n}},\ \
n\ge 1,\end{equation} it follows from \eqref{B3} that
\begin{equation}\label{B4}
\limsup_{n\to\8}\ff{ \sum_{l=1}^nf(X_l^\xi)  }{\ss{2n\log\log
n}}= \hat D_f \limsup_{n\to\infty} \LL_n^{f,\xi}(1)\le\hat D_f,~~~\P\mbox{-a.s.}
\end{equation}
On the other hand, since the limits points of
$(\Lambda_n^{f,\xi}(t))$ coincides with $\mathcal {H}$ and
$h\in\mathcal {H}$ with $h(t)=t,t\in[0,1]$, there exists a
subsequence $n_k\uparrow\8$ as $k\to\8$ such that $\P$-a.s.
\begin{equation*}
\lim_{k\to\infty} \sup_{t\in[0,1]}|\Lambda_{n_k}^{f,\xi}(t)-h(t)|= 0.
\end{equation*} In particular,
combining this with \eqref{B1} for    $k=n-1$ and $t=\ff{k}{n}$, we
deduce    $\P$-a.s.
$$
\lim_{k\to\8}\ff{\sum_{l=1}^{n_k}f(X_l^\xi) }{\ss{2n_k\log\log
n_k}}=\lim_{k\to\8} \hat D_f   \LL_{n_k}^{f,\xi}(1)=\hat D_f,
$$
which together with  \eqref{B4}   yields
\begin{equation*}
\limsup_{n\to\8}\ff{\sum_{l=1}^nf(X_l^\xi)}{\ss{2n\log\log
n}}=\hat D_f ~~~\P\mbox{-a.s.}
\end{equation*}
Replacing $f$ by $-f$, this formula reduces to
$$
\liminf_{n\to\8}\ff{\sum_{l=1}^nf(X_l^\xi)}{\ss{2n\log\log
n}}=-\hat D_f ~~~\P\mbox{-a.s.}
$$
   Therefore, \eqref{B2} holds.
\end{proof}

%Observe that
%\begin{equation*}
%\begin{split}
%\ff{\int_0^tf(X_s^\xi)\d s}{\ss{2t\log\log t}}&=\ss{\ff{\lfloor
%t\rfloor\log\log \lfloor t\rfloor}{t\log\log
%t}}\cdot\ff{\int_0^{\lfloor t\rfloor }f(X_s^\xi)\d s}{\ss{2\lfloor
%t\rfloor\log\log \lfloor t\rfloor}}\\
%&\quad+\ff{\int_0^{\lfloor t\rfloor} \{f(X_s^\xi)-f(X_{\lfloor
%s\rfloor}^\xi)\}\d s}{\ss{2t\log\log t}} +\ff{\int_{\lfloor
%t\rfloor}^t  f(X_s^\xi) \d s}{\ss{2t\log\log t}},~~~~f\in
%C_{p,\gg}(\C).
%\end{split}
%\end{equation*}
%By \eqref{B2}, it is sufficient to show
%\begin{equation*}
%\ff{\int_0^{\lfloor t\rfloor} \{f(X_s^\xi)-f(X_{\lfloor
%s\rfloor}^\xi)\}\d s}{\ss{2t\log\log
%t}}\to0,~~\P-\mbox{a.s.},~~\ff{\int_{\lfloor t\rfloor}^t  f(X_s^\xi)
%\d s}{\ss{2t\log\log t}}\to0,~~\P-\mbox{a.s.}
%\end{equation*}

\end{document}